\documentclass[11pt,a4paper]{article}

\usepackage[
    left=27mm,
    right=27mm,
    top=24mm,
    bottom=25mm
]{geometry}

\usepackage[
    backend=biber,
    style=numeric,
    sorting=nyt,
    giveninits=true,
    maxbibnames=99,
    doi=true,
    url=true,
    isbn=false
]{biblatex}
\usepackage{mathtools}
\allowdisplaybreaks

\usepackage{amssymb}

\usepackage[mathscr]{eucal}

\usepackage{mismath}

\usepackage{bm}

\usepackage{graphicx}
\graphicspath{{figs/}}

\usepackage{tikz}
\usetikzlibrary{calc,math}

\usepackage{subcaption}

\usepackage{multirow}
\usepackage{makecell}
\setcellgapes{2pt}

\usepackage{booktabs}

\usepackage{siunitx}
\usepackage{algorithm}
\usepackage{algpseudocode}

\usepackage{relsize}

\newcommand*{\proofheading}[1]{%
  \smallskip\noindent\emph{#1}\nobreak\enspace\ignorespaces
}

\DeclareMathOperator{\Dist}{dist}
\DeclareMathOperator{\Span}{span}

\DeclareMathOperator{\Supp}{supp}

\DeclareMathOperator{\Diam}{diam}

\DeclarePairedDelimiter{\floor}{\lfloor}{\rfloor}

\usepackage[skins]{tcolorbox}

\newtcolorbox{justabox}[2][]{%
  enhanced,
  attach boxed title to top center={
    yshift=-3mm,
    yshifttext=-1mm
  },
  colframe=blue!75!black,
  colbacktitle=red!80!black,
  fonttitle=\bfseries,
  title={#2},
  #1
}

\DeclareSymbolFont{wasycustom}{U}{wasy}{m}{n}
\DeclareMathSymbol{\ocircle}{\mathbin}{wasycustom}{"23}

\usepackage{hyperref}

\usepackage{amsthm}
\usepackage[capitalise,nameinlink,noabbrev]{cleveref}
\crefname{equation}{}{}
\Crefname{equation}{}{}
\hypersetup{
    colorlinks=true,
    linkcolor=blue!45!black,
    citecolor=blue!45!black,
    urlcolor=blue!45!black,
    pdfauthor={Changqing Ye},
    pdftitle={An Order-One Lower Bound on the Error of Scalable Generalized Multiscale Finite Element Space Constructions},
    pdfsubject={An order-one error lower bound for scalable generalized multiscale finite element space constructions over rough coefficients},
    pdfkeywords={generalized finite element method, multiscale finite element method, numerical homogenization, rough coefficients, localization, error lower bounds}
}

\theoremstyle{plain}
\newtheorem{theorem}{Theorem}[section]
\newtheorem{lemma}[theorem]{Lemma}
\newtheorem{proposition}[theorem]{Proposition}

\theoremstyle{definition}
\newtheorem{definition}[theorem]{Definition}

\theoremstyle{remark}
\newtheorem{remark}[theorem]{Remark}

\numberwithin{equation}{section}

\newcommand{\Torus}{\mathbb{T}^d}
\newcommand{\Per}{\mathrm{per}}

\newcommand{\shorttitle}{Lower Bound for Scalable Multiscale Spaces}

\title{An Order-One Lower Bound on the Error of Scalable Generalized Multiscale Finite Element Space Constructions}
\author{Changqing Ye\thanks{Corresponding author.
Email:
\href{mailto:ye_changqing@outlook.com}{ye\_changqing@outlook.com}.}\\[2pt]
\small Laboratory of Computational Physics,\\[-1pt]
\small Institute of Applied Physics and Computational Mathematics, Beijing, China}
\date{}

\begin{document}

\pagestyle{myheadings}
\markboth{C. Ye}{\shorttitle}
\maketitle

\begin{abstract}
Several coefficient-adapted methods provide optimal-order approximation for elliptic equations with rough coefficients.
Prominent examples include localized orthogonal decomposition, multiscale spectral GFEM, and constraint energy-minimizing GMsFEM.
Their proven accuracy, however, is obtained by allowing the localization radius or the local spectral dimension to grow as the coarse scale \(H\) tends to zero.
Classical MsFEM has an FEM-like local construction, but its available analysis does not give a coefficient-uniform \(\bigO(H)\) energy estimate over the full bounded-contrast measurable coefficient class.
Motivated by this gap, we formalize an FEM-like notion of structural scalability.
A chosen spatially local basis has uniformly bounded overlap, hence \(\bigO(1)\) stiffness entries per row, and every anchored local span uses coefficient information from only \(\bigO(1)\) coarse-element layers.
We prove that no deterministic construction satisfying fixed bounds on the support radius, coefficient-information radius, and local multiplicity can converge uniformly over the coefficient class.
In fact, its worst-case \(L^2\)-to-energy Galerkin error remains bounded below by a positive constant independent of \(H\).
The lower bound is established using a fixed finite family of smooth periodic coefficients and smooth right-hand sides.
The proof combines coefficients that coincide on local patches, a finite-dimensional approximation lower bound for corrector fields, a positive-density mesh argument, and strong periodic corrector convergence.
Thus uniform optimal accuracy requires at least one local construction parameter to grow or requires coefficient information beyond the fixed-visibility model.
\end{abstract}



\section{Introduction}

We consider the Dirichlet problem
\begin{equation}
    -\nabla\cdot(\kappa\nabla u)=f \quad\text{in }\Omega, \qquad u=0 \quad\text{on }\partial\Omega,
    \label{eq:introduction-pde}
\end{equation}
where \(\kappa\) is measurable and satisfies \(1\le\kappa\le\rho\) for a fixed \(\rho>1\).
Without further regularity of \(\kappa\), standard finite element error analysis cannot use a coefficient-uniform \(H^2\)-regularity estimate.
Several multiscale constructions nevertheless recover optimal-order energy approximation by building coefficient-adapted trial spaces.
Their approximation results raise the central question of this paper.
Can the same accuracy be achieved while the localization radius, local dimension, and coefficient-information radius remain bounded under refinement?

To place this question in context, we review four representative multiscale constructions and the approximation results available for them.
The comparison focuses on whether the proven localization depth and local spectral dimension remain bounded as \(H\to0\).

\paragraph{LOD and SLOD.}
Localized orthogonal decomposition splits a standard coarse space from a coefficient-dependent fine-scale kernel and corrects the coarse functions by energy-orthogonal fine-scale solves.
The ideal correctors are global.
Exponential decay permits their truncation to oversampling patches, and \(\bigO(H)\) energy accuracy follows when the number of coarse-element layers is chosen of order \(\lvert\log H\rvert\) \cite{Maalqvist2014}.
Thus LOD gives coefficient-robust optimal order, but its proven localization radius grows as \(H\to0\).
In computations, the rough coefficient and the corrector problems are typically resolved on a fine mesh of size \(h\), which is held fixed while \(H\) is varied within the two-level regime \(h\ll H\).
If an LOD corrector uses \(k\) coarse layers, its patch contains \(\bigO(k^d)\) coarse elements and, on quasi-uniform meshes, approximately \(\bigO((kH/h)^d)\) fine-grid unknowns.
Compared with a fixed-layer local solve on the same fine grid, the local problem is therefore larger by a factor of order \(k^d\), which becomes \(\lvert\log H\rvert^d\) when \(k\simeq\lvert\log H\rvert\).
The comparison is with an elementwise construction on a fixed number of coarse layers.
Measured in coarse elements, the size of an LOD corrector problem grows like \(\lvert\log H\rvert^d\), so the local construction stencil is not bounded uniformly under coarse refinement.
The Super-Localized Orthogonal Decomposition (SLOD) modifies the localized right-hand sides to obtain much faster observed decay \cite{Hauck2023}.
Under the spectral-geometric condition used in its super-localization argument, a localization radius of order \(H\lvert\log H\rvert^{(d-1)/d}\) is sufficient, improving on the LOD radius \(H\lvert\log H\rvert\) but still using a growing number of coarse layers.
The full super-exponential decay is not presently an unconditional coefficient-uniform theorem; the unconditional a priori analysis available for SLOD-type constructions recovers at least the classical LOD localization rate \cite{Freese2024}.

\paragraph{CEM-GMsFEM.}
CEM-GMsFEM first solves local spectral problems to identify auxiliary modes associated with high-conductivity channels and then constructs energy-minimizing basis functions subject to constraints imposed by those modes \cite{Chung2018}.
The localized minimizations are performed on oversampling regions.
When the auxiliary space contains the required small-eigenvalue modes and the oversampling depth is suitably chosen, the method has an energy error proportional to \(H\), with constants independent of the contrast in the regime covered by the analysis.
The original localization theorem uses an oversampling depth of order \(\log(\max\kappa/H)\); for the bounded-contrast class considered here, this still grows like \(\lvert\log H\rvert\).

\paragraph{Spectral generalized finite element methods.}
These methods instead reduce local approximation spaces by spectral compression.
MS-GFEM constructs local \(A\)-harmonic approximation spaces on oversampling regions, selects leading modes of compact transfer or restriction operators, and couples the local spaces by a partition of unity \cite{Babuska2011}.
A recent unified analysis proves local errors of order \(\exp(-cn^{1/d})\), under its abstract assumptions, for a broad class of multiscale problems \cite{Ma2026}.
For the elliptic problem considered here, this gives \(\bigO(H)\) global energy accuracy with a sufficient local spectral dimension \(n\) of order \(\lvert\log H\rvert^d\).
Related edge-based and ring-spectral constructions use the same compression principle \cite{Hou2016,Alber2025}.
The Super-Localized Generalized Finite Element Method combines SLOD-type snapshots with a partition of unity and local spectral compression \cite{Freese2024}.
It can keep the support patches fixed by increasing the number of basis functions per coarse entity.
Thus these spectral methods trade a growing localization radius for a growing local spectral dimension.

\paragraph{Classical MsFEM.}
This method is closer to the local architecture of standard FEM.
Its basis functions are computed by coefficient-dependent harmonic problems on individual coarse elements and are coupled through coarse nodal data \cite{Hou1997}.
Its basis has fixed coarse support and uses fixed-patch coefficient information.
In the periodic setting, however, the energy estimate contains a resonance contribution; for example, Hou, Wu, and Cai obtain a term of order \((\varepsilon/H)^{1/2}\) in addition to the coarse discretization error \cite[Theorem~5.1]{Hou1999}.
Oversampling reduces the effect of artificial local boundary data, but the classical theory does not provide a coefficient-uniform \(\bigO(H)\) estimate over the full measurable class considered here.
Fixed-support optimal estimates are known for special rough subclasses, such as unidirectional coefficients \cite{Babuska1994}, but not for general \(\kappa\in\mathcal K_\rho\).

The preceding methods establish that optimal approximation of elliptic problems with rough coefficients is possible, but their proven constructions do not have the same refinement-independent local complexity bounds as low-order FEM.
On a shape-regular mesh, a standard FEM basis has \(\bigO(1)\) functions per element, \(\bigO(1)\) overlap, and hence \(\bigO(1)\) nonzero stiffness entries per row; each basis function is also determined from \(\bigO(1)\) neighboring coarse elements.
We use these two features to define \emph{FEM-like structural scalability} for a coefficient-adapted construction:
\begin{enumerate}
    \item a chosen spatially local basis has fixed support radius and fixed local multiplicity, which imply \(\bigO(1)\) stiffness entries per row;
    \item each anchored local span depends only on the coefficient in a fixed number of coarse-element layers.
\end{enumerate}
The second condition means dependence on coefficient data from \(\bigO(1)\) coarse-element layers, not dependence on only finitely many scalar samples of the coefficient.
The rule may use the entire restriction of the measurable coefficient to each coefficient-information patch and may process these data in an arbitrarily nonlinear, discontinuous, or noncomputable way.
Matrix sparsity alone would not be an intrinsic condition, because a global change of basis can diagonalize the stiffness matrix; the chosen basis must also be spatially local.

Formally, every basis function is assigned to an anchor element, is supported within \(m\) layers of that anchor, and shares the anchor with at most \(C_{\mathrm{loc}}\) basis functions.
The anchored span may depend on the coefficient only through its restriction to the corresponding \((m+k)\)-layer patch.
The parameters \(m\), \(k\), and \(C_{\mathrm{loc}}\) are fixed independently of \(H\).
This fixed-visibility requirement distinguishes local construction from support locality alone.
A compactly supported function may still have been designed using coefficient information from the whole domain.

With these notions in place, the main result can be stated informally as follows.
No deterministic FEM-like scalable construction in the preceding sense converges uniformly over the full bounded-contrast coefficient class.
More strongly, for every such construction and every sufficiently fine scale, there is a coefficient for which the normalized worst-case energy error is bounded below by a positive constant independent of \(H\).
Thus the failure is not merely a loss of the optimal \(\bigO(H)\) rate.
The worst-case error does not tend to zero.

The lower bound uses a fixed finite family.
Before the scale and the construction rule are specified, the proof fixes finitely many smooth periodic coefficient profiles and smooth compactly supported right-hand sides.
At every sufficiently small \(H\), at least one pair consisting of a coefficient profile and a right-hand side gives the stated error lower bound.
The coefficient period is \(\varepsilon_H=LH\), where \(L>1\) is fixed after the coefficient-information radius; thus the coefficient period remains proportional to the coarse mesh size, which coincides with the resonance regime of classical MsFEM.

The proof has four steps.
\begin{enumerate}
    \item Support locality and bounded anchor multiplicity give a uniform dimension bound for the restriction of the selected space to one mesh element.
        Local dependence on the coefficient makes that restriction identical for coefficients agreeing on a slightly larger patch.
    \item Small coefficient perturbations placed outside an open core of a periodicity cell generate arbitrarily many linearly independent corrector fields inside the core.
        This produces a finite family of coefficients that agree on the relevant local patches but whose solution gradients cannot all be approximated by a single space of fixed dimension.
    \item On an arbitrary quasi-uniform, shape-regular simplicial mesh, a positive fraction of the elements have their coefficient-information patches contained in periodically repeated cores.
    \item A compactly supported homogenized solution, affine on an interior cube, realizes the corrector fields in exact solutions with smooth right-hand sides in \(L^2(\Omega)\).
        Strong corrector convergence transfers the finite corrector family to exact solutions, and summation over a positive fraction of the mesh gives an order-one global Galerkin error.
\end{enumerate}

The theorem is an approximation lower bound under a restriction on local coefficient dependence, not a runtime or storage lower bound.
It concerns deterministic rules with exact fixed-visibility consistency and the minimax order \(\inf_{\mathscr M}\sup_\kappa\).
It does not settle the support-only problem with order \(\sup_\kappa\inf_{V_H}\), where a compactly supported basis may be redesigned after the whole coefficient is known.
That globally informed problem remains unresolved.

The remainder of the paper is organized as follows.
\cref{sec:formulation} formulates the information model and states the main theorem.
\cref{sec:local-rank} derives the local dimension bound implied by fixed visibility.
\crefrange{sec:finite-family}{sec:homogenization} construct the finite corrector family, identify a positive-density set of elements whose coefficient-information patches are contained in the cores, and realize the corrector fields through exact solutions with smooth right-hand sides.
\cref{sec:main-proof} combines these results.
Finally, \cref{sec:scope} returns to the scalability consequence, quantifier order, and boundary with globally informed constructions.

\section{Problem formulation and main result}
\label{sec:formulation}

\paragraph{Notation.}
We use \(\mathbb{N}_0\coloneqq\{0,1,2,\dots\}\), write \(M\Subset S\) when \(\overline M\) is a compact subset of \(S\), and use \(\lvert E\rvert\) and \(\#\mathcal F\) for the Lebesgue measure of \(E\) and the cardinality of a finite set \(\mathcal F\), respectively.
If \(X\) is a Hilbert space and \(Z\subset X\), then
\[
    \Dist_X(v,Z) \coloneqq \inf_{z\in Z}\norm{v-z}_X.
\]
By contrast, \(\Dist(x,E)\) denotes Euclidean point-to-set distance.
The notation \([w]\), used in the Poincar\'e inequality below, denotes the equivalence class of \(w\) modulo constants, and we set \(r_+\coloneqq\max\{r,0\}\).
After introducing the periodicity cell \(Y\), we use the subscript \(\Per\) to denote \(Y\)-periodicity.
We write \(a\lesssim b\) when \(a\le cb\) with a constant \(c\) independent of the parameters specified in the relevant statement.

\subsection{Elliptic solution operator and coarse meshes}

Let \(d\ge2\), let \(\Omega\subset\mathbb R^d\) be a bounded, connected, polyhedral Lipschitz domain, and fix \(\rho>1\).
Define
\[
    \mathcal K_\rho \coloneqq \bigl\{ \kappa\in L^\infty(\Omega) \bigm| 1\le\kappa\le\rho\ \text{a.e. in }\Omega \bigr\}.
\]
For \(\kappa\in\mathcal K_\rho\), set
\[
    a_\kappa(u,v) \coloneqq \int_\Omega\kappa\nabla u\cdot\nabla v\di x, \qquad \norm{v}_{a_\kappa} \coloneqq a_\kappa(v,v)^{1/2}.
\]
For \(f\in L^2(\Omega)\), Lax--Milgram gives a unique \(u_{\kappa,f}\in H_0^1(\Omega)\) satisfying
\begin{equation}
    a_\kappa(u_{\kappa,f},v) = (f,v)_{L^2(\Omega)} \qquad \bigl(v\in H_0^1(\Omega)\bigr).
    \label{eq:weak-problem}
\end{equation}

Let \(\mathcal H\subset(0,\infty)\) have \(0\) as an accumulation point.
For every \(H\in\mathcal H\), let \(\mathcal T_H\) be a conforming simplicial mesh of \(\Omega\), indexed so that
\[
    H\coloneqq\max_{T\in\mathcal T_H}\Diam(T).
\]
We assume uniform shape regularity and quasi-uniformity.
Thus there are fixed \(\gamma_{\mathrm{sh}}\ge1\) and \(c_{\mathrm{qu}}\in(0,1]\) such that
\begin{equation}
    \frac{\Diam(T)}{r_T}\le\gamma_{\mathrm{sh}}, \qquad c_{\mathrm{qu}}H\le\Diam(T)\le H \qquad \bigl(T\in\mathcal T_H\bigr),
    \label{eq:mesh-regularity}
\end{equation}
where \(r_T\) is the inradius of \(T\).
Two elements are neighbors if they share a codimension-one face.
For \(j\in\mathbb N_0\), \(\omega_j(T)\) denotes the union of all elements at face-adjacency graph distance at most \(j\) from \(T\).
Supports below are essential closed supports in \(\overline\Omega\).

\subsection{Locally supported spaces with element assignments}

We next separate the geometric support restriction from the coefficient-information restriction.
The first definition records only the trial space, its chosen basis, and the assignment of basis functions to coarse elements.
This anchored structure is needed because the later locality condition is imposed anchor by anchor, not merely on the global span.

\begin{definition}[Locally supported trial space with an element assignment]
\label{def:support-local}
Fix \(H\in\mathcal H\), \(m\in\mathbb N_0\), and \(C_{\mathrm{dim}},C_{\mathrm{loc}}>0\).
A locally supported trial space with an element assignment and parameters \((m,C_{\mathrm{dim}},C_{\mathrm{loc}})\) consists of
\[
    \begin{aligned}
        V_H&\subset H_0^1(\Omega), &\qquad \Psi_H&=\{\psi_{H,i}\}_{i=1}^{N_H},\\
        \alpha_H&\colon \{1,\dots,N_H\}\longrightarrow\mathcal T_H,
    \end{aligned}
\]
with the following properties:
\begin{enumerate}
    \item \(\Psi_H\) is a basis of \(V_H\);
    \item \(N_H=\dim V_H\le C_{\mathrm{dim}}H^{-d}\);
    \item with \(T_{H,i}\coloneqq\alpha_H(i)\), \(\Supp\psi_{H,i}\subseteq\omega_m(T_{H,i})\);
    \item for every \(T\in\mathcal T_H\), \(\#\{i\mid T_{H,i}=T\}\le C_{\mathrm{loc}}\).
\end{enumerate}
The map \(\alpha_H\) is called the \emph{anchor assignment}, and \(T_{H,i}\) is the \emph{anchor element} of \(\psi_{H,i}\).
The zero-dimensional choice \(V_H=\{0\}\), with empty basis and anchor map, is allowed.
\end{definition}

Under quasi-uniformity, the qualitative requirement \(N_H=\bigO(H^{-d})\) already follows from bounded anchor multiplicity:
\[
    N_H \le \floor{C_{\mathrm{loc}}}\,\#\mathcal T_H \lesssim H^{-d}.
\]
Thus the existence of some uniform global dimension constant is redundant.
The prescribed constant \(C_{\mathrm{dim}}\) is retained in \cref{def:support-local} so that the admissibility data match the locally supported problem discussed in \cref{sec:scope}.

The anchor is an indexing device, not an interpolation node or an assumption on how a function is computed.
It identifies the coarse element assigned to a basis function.
A basis function anchored at \(T\) may extend over all of \(\omega_m(T)\), and at most \(\floor{C_{\mathrm{loc}}}\) basis functions may share that anchor.
The anchor assignment is part of the anchored structure; the same space with a different basis or assignment need not satisfy the same locality condition.

The same assumptions give the row-sparsity part of FEM-like scalability.
A stiffness entry can be nonzero only if the two corresponding supports overlap on a set of positive measure.
The two anchor elements are then separated by at most \(2m\) face-adjacency layers.
Uniform mesh regularity and bounded anchor multiplicity therefore bound the number of nonzero stiffness entries in every row by a constant depending only on \(d\), \(m\), and \(C_{\mathrm{loc}}\), independently of \(H\) and \(\kappa\).
The converse is false because a sparse stiffness matrix alone does not imply spatially local basis functions.

The second definition introduces the coefficient-to-space rule.
It formalizes the requirement that the subspace attached to an anchor can depend only on the coefficient in a prescribed enlarged patch.
This is the fixed-visibility assumption used later to obtain a common local restriction space.

\begin{definition}[Fixed-visibility selection rule]
\label{def:rule}
Fix \(H\in\mathcal H\), \(m,k\in\mathbb N_0\), and \(C_{\mathrm{dim}},C_{\mathrm{loc}}>0\).
A deterministic fixed-visibility selection rule \(\mathscr M\), with parameters \((H,m,k,C_{\mathrm{dim}},C_{\mathrm{loc}})\), maps every \(\kappa\in\mathcal K_\rho\) to a locally supported trial space with an element assignment
\[
    \mathscr M(\kappa) \coloneqq \bigl( V_H^\mathscr M(\kappa), \Psi_H^\mathscr M(\kappa), \alpha_H^\mathscr M(\kappa) \bigr),
\]
where
\[
    \begin{aligned}
        N_H^\mathscr M(\kappa) &\coloneqq \dim V_H^\mathscr M(\kappa), &\qquad \Psi_H^\mathscr M(\kappa) &= \left\{ \psi_{H,i}^\mathscr M(\kappa) \right\}_{i=1}^{N_H^\mathscr M(\kappa)},\\
        \alpha_H^\mathscr M(\kappa) &\colon \left\{1,\dots,N_H^\mathscr M(\kappa)\right\} \longrightarrow \mathcal T_H, &\qquad T_{H,i}^\mathscr M(\kappa) &\coloneqq \alpha_H^\mathscr M(\kappa)(i).
    \end{aligned}
\]
If \(\mathscr M\) is fixed, the rule label is omitted and the resulting quantities are denoted by
\[
    \begin{aligned}
        V_H^\kappa &\coloneqq V_H^\mathscr M(\kappa), &\qquad N_H^\kappa &\coloneqq \dim V_H^\kappa,\\
        \Psi_H^\kappa &\coloneqq \Psi_H^\mathscr M(\kappa), &\qquad \psi_{H,i}^\kappa &\coloneqq \psi_{H,i}^\mathscr M(\kappa),\\
        \alpha_H^\kappa &\coloneqq \alpha_H^\mathscr M(\kappa), &\qquad T_{H,i}^\kappa &\coloneqq T_{H,i}^\mathscr M(\kappa),
    \end{aligned}
\]
where the superscript \(\kappa\) records coefficient dependence.
For \(T\in\mathcal T_H\), define the anchored subspace at \(T\) by
\[
    W_{H,T}^\kappa \coloneqq \Span \bigl\{ \psi_{H,i}^\kappa \bigm| T_{H,i}^\kappa=T \bigr\} \subset H_0^1(\Omega).
\]
The rule has visibility enlargement \(k\) if, for all \(\kappa,\widetilde\kappa\in\mathcal K_\rho\) and \(T\in\mathcal T_H\),
\begin{equation}
    \kappa=\widetilde\kappa \ \text{a.e. on }\omega_{m+k}(T) \quad\Longrightarrow\quad W_{H,T}^\kappa = W_{H,T}^{\widetilde\kappa}.
    \label{eq:visibility-consistency}
\end{equation}
\end{definition}

The selected structure is independent of \(f\).
No continuity, measurability, computability, finite-precision, translation-covariance, or setup-cost condition is imposed on \(\mathscr M\).
The support patch of a function anchored at \(T\) is \(\omega_m(T)\), and the rule may additionally depend on coefficient data from \(k\) surrounding element layers.
Accordingly, its coefficient-information patch is \(\omega_{m+k}(T)\).
For a fixed mesh and scale, the local coefficient data at an anchor \(T\) are the almost-everywhere equivalence class of the coefficient on \(\omega_{m+k}(T)\), together with the fixed mesh geometry.
The condition in \cref{eq:visibility-consistency} imposes locality on the anchored subspace rather than on a normalized or ordered basis.
It is invariant under changes of basis within one anchored span.
Mixing basis functions assigned to different anchors changes the anchored output structure and is not an invariance of the model, even if the global trial space is unchanged.
Local patch solves and local spectral constructions satisfy the condition when their boundary data, normalizations, and mode-selection criteria depend only on the coefficient-information patch; global constraints or globally computed boundary data need not satisfy it.
The theorem concerns deterministic rules.
Randomized rules require a separate expected- or high-probability formulation and are not covered.
Likewise, a tolerance-dependent construction leaves the fixed-parameter class if its retained mode count or coefficient-information radius grows as \(H\to0\).
No regularity, continuity, or computability of the coefficient-to-space map is assumed.

For a fixed finite-dimensional trial space \(V_H\), let \(u_{\kappa,f,H}\in V_H\) denote the Galerkin solution
\[
    a_\kappa(u_{\kappa,f,H},v_H) = (f,v_H)_{L^2(\Omega)} \qquad \bigl(v_H\in V_H\bigr).
\]
We define its \(L^2\)-to-energy error by
\begin{equation}
    \mathcal E_H(\kappa;V_H) \coloneqq
    \sup_{f\in L^2(\Omega)\setminus\{0\}}
    \frac{
        \norm{u_{\kappa,f}-u_{\kappa,f,H}}_{a_\kappa}
    }{
        \norm{f}_{L^2(\Omega)}
    }.
    \label{eq:operator-error}
\end{equation}

The main theorem states the lower bound in the normalization of \cref{eq:operator-error}.
The statement is uniform in the selection rule and in the mesh family; only the constants may depend on the fixed structural parameters.
\begin{theorem}[Order-one error lower bound under scalable fixed visibility]
\label{thm:main}
Let \(d\ge2\), let \(\Omega\subset\mathbb R^d\) be a bounded, connected, polyhedral Lipschitz domain, and fix \(\rho>1\), \(m,k\in\mathbb N_0\), and \(C_{\mathrm{dim}},C_{\mathrm{loc}}>0\).
Let \(\{\mathcal T_H\}_{H\in\mathcal H}\) be a mesh family satisfying \cref{eq:mesh-regularity}.
There exist constants \(c_*,H_*>0\) such that, for every \(H\in\mathcal H\) with \(0<H<H_*\) and every fixed-visibility rule \(\mathscr M\) with parameters \((H,m,k,C_{\mathrm{dim}},C_{\mathrm{loc}})\), there exists a coefficient \(\kappa_H\in\mathcal K_\rho\) for which
\begin{equation}
    \mathcal E_H\bigl(\kappa_H;V_H^\mathscr M(\kappa_H)\bigr)\ge c_*.
    \label{eq:main-lower-bound}
\end{equation}
Consequently,
\begin{equation}
    \liminf_{\substack{H\to0\\H\in\mathcal H}}
    \inf_{\mathscr M}
    \sup_{\kappa\in\mathcal K_\rho}
    \mathcal E_H\bigl(\kappa;V_H^\mathscr M(\kappa)\bigr)
    \ge c_*>0,
    \label{eq:minimax-lower-bound}
\end{equation}
where the infimum is over the rules in \cref{def:rule} with the fixed parameters above.
The constants \(c_*\) and \(H_*\) may depend only on \(d,\rho,m,k,C_{\mathrm{loc}},\Omega,\gamma_{\mathrm{sh}},c_{\mathrm{qu}}\); in particular, they are independent of \(C_{\mathrm{dim}}\), \(H\), the rule, and the particular mesh family satisfying the stated regularity bounds.
\end{theorem}

The proof establishes the stronger finite-family result in \cref{prop:finite-family-realization}: a single finite collection of smooth periodic coefficient profiles and smooth compactly supported right-hand sides, chosen before \(H\) and before the rule, contains a member that yields the lower bound for every sufficiently small \(H\) and every admissible rule.

\begin{remark}
The order-one scale in \cref{thm:main} is sharp for the normalized error in \cref{eq:operator-error}.
Indeed, if \(C_\mathrm{P}\) is the Poincar\'e constant of \(\Omega\), then testing \cref{eq:weak-problem} with \(u_{\kappa,f}\), using \(\kappa\ge1\), and applying Poincar\'e's inequality gives \(\norm{u_{\kappa,f}}_{a_\kappa}\le C_\mathrm{P}\norm{f}_{L^2(\Omega)}\).
Galerkin best approximation and the admissible competitor \(0\in V_H\) therefore imply \(\mathcal E_H(\kappa;V_H)\le C_\mathrm{P}\), for every coefficient and every conforming trial space.
Thus \cref{thm:main} gives a \(\Theta(1)\) worst-case lower bound in this normalization.
The upper scale comes from comparison with the zero space; it is not a constructive approximation method.
\end{remark}

\section{From fixed visibility to a common local space}
\label{sec:local-rank}

The first step converts support locality and bounded anchor multiplicity into a local dimension bound.
For a fixed element \(K\), only basis functions anchored within \(m\) graph layers can affect \(K\).
The parameter \(q\) below bounds the dimension of their restrictions to \(K\), and \(R\) is the radius on which coefficient agreement is sufficient to make those restrictions identical.

Let
\[
    r\coloneqq\floor{C_{\mathrm{loc}}}, \qquad \mathcal N_m(K) \coloneqq \bigl\{ T\in\mathcal T_H \bigm| \Dist_{\mathcal T_H}(T,K)\le m \bigr\},
\]
where \(\Dist_{\mathcal T_H}\) is the face-adjacency graph distance.
Since every \(d\)-simplex has at most \(d+1\) face neighbors, we have the mesh-independent bound
\[
    b_{d,m} \coloneqq \sum_{j=0}^m(d+1)^j \ge \sup_{H\in\mathcal H} \sup_{K\in\mathcal T_H} \#\mathcal N_m(K).
\]
We therefore set
\begin{equation}
    q\coloneqq b_{d,m}r, \qquad R\coloneqq 2m+k.
    \label{eq:q-and-R}
\end{equation}
The notation \(\omega_R(K)\) denotes the \(R\)-layer element patch defined above.

The lemma records this reduction in a form that can be used later on elements whose coefficient-information patches lie in regions where the candidate coefficients coincide.

\begin{lemma}[Common local restriction space]
\label{lem:restriction}
Fix a rule \(\mathscr M\), a coefficient \(\kappa\in\mathcal K_\rho\), and an element \(K\in\mathcal T_H\).
Using the fixed-rule convention above, write \(V_H^\kappa\) for \(V_H^\mathscr M(\kappa)\) and define
\[
    Y_K^\kappa \coloneqq \bigl\{ [v\vert_K] \bigm| v\in V_H^\kappa \bigr\} \subset H^1(K)/\mathbb R.
\]
Then
\begin{equation}
    \dim Y_K^\kappa\le q.
    \label{eq:local-rank-bound}
\end{equation}
If \(\kappa,\widetilde\kappa\in\mathcal K_\rho\) satisfy \(\kappa=\widetilde\kappa\) a.e. on \(\omega_R(K)\), then
\begin{equation}
    Y_K^\kappa = Y_K^{\widetilde\kappa}.
    \label{eq:common-restriction-space}
\end{equation}
\end{lemma}

\begin{proof}
Suppose that a basis function anchored at \(T\) is nonzero on \(K\) on a set of positive measure.
Its support lies in \(\omega_m(T)\), so \(T\in\mathcal N_m(K)\).
At most \(r\) basis functions are anchored at each such \(T\).
Restriction to \(K\) and passage to the quotient cannot increase dimension, which proves \cref{eq:local-rank-bound}.

For \(T\in\mathcal N_m(K)\), the triangle inequality for graph distance gives
\[
    \omega_{m+k}(T) \subseteq \omega_{2m+k}(K) = \omega_R(K).
\]
Hence \cref{eq:visibility-consistency} makes every anchored subspace that can contribute nontrivially on \(K\) identical for \(\kappa\) and \(\widetilde\kappa\).
The sums of their restriction spaces are therefore equal, proving \cref{eq:common-restriction-space}.
\end{proof}

The corrector-family construction in the next section is needed only when \(q\ge1\).
If \(q=0\), then \(r=0\), no basis function can be assigned to any anchor, and every selected trial space equals \(\{0\}\).
Choose any nonzero \(f_0\in C_{\mathrm c}^\infty(\Omega)\), set \(a_0\equiv1\), and fix any \(L>1\).
For every \(H\), the corresponding coefficient \(\kappa_{H,0}=a_0(\cdot/(LH))\) is identically one, whereas the exact solution \(u_{1,f_0}\) is nonzero.
Since the Galerkin solution in the zero space vanishes, one may take
\[
    c_* = \frac{\norm{\nabla u_{1,f_0}}_{L^2(\Omega)}}{\norm{f_0}_{L^2(\Omega)}}>0,
    \qquad H_*=1.
\]
This proves \cref{prop:finite-family-realization} for \(q=0\); hence its proof below may assume \(q\ge1\).

\section{A finite corrector family with common local coefficient data}
\label{sec:finite-family}

The preceding section reduces fixed visibility to a common \(q\)-dimensional local approximation space on any element whose coefficient-information patch lies in a region where the coefficients coincide.
We now establish the approximation result used in the lower bound.
We construct a finite family of smooth periodic coefficients that are all equal to one on a core \(D\) and therefore have the same local coefficient data there.
Their corrector fields in the first coordinate direction span \(q+1\) independent directions inside the same core.
The construction uses exterior dipoles to generate independent harmonic fields in \(D\) and a scalar-coefficient perturbation argument to realize those fields as corrector gradients.
The next section supplies many mesh elements whose coefficient-information patches lie in copies of \(D\), and \cref{sec:homogenization} then realizes the cell correctors through exact PDE solutions with smooth right-hand sides.

Let \(Y\coloneqq(-1/2,1/2)^d\).
Identifying opposite faces, we regard \(Y\) as the flat torus \(\Torus\).
Choose concentric \emph{open cubes}
\begin{equation}
    D_-\Subset D_0\Subset D\Subset Y.
    \label{eq:nested-cubes}
\end{equation}
\cref{fig:nested-cell-geometry} summarizes the two roles of these sets.
The perturbations that distinguish the coefficients are placed outside \(D\), whereas the mesh elements used in the lower bound and their coefficient-information patches remain inside \(D\).
The figure is only a two-dimensional schematic; the proof below uses the stated separation distances in arbitrary dimension.

\begin{figure}[htbp]
    \centering
    \includegraphics[width=0.88\textwidth]{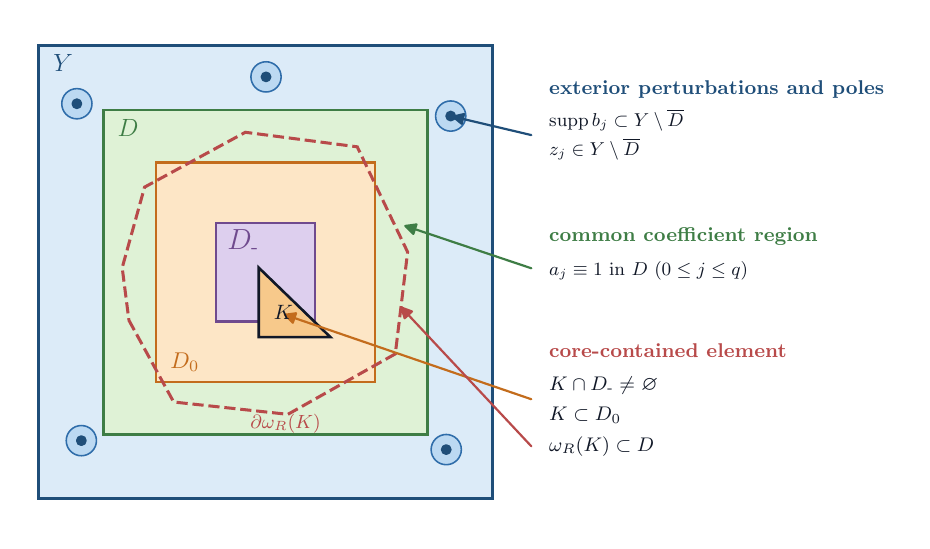}
    \caption{Reference-cell geometry behind the construction.
    The concentric squares represent the cubes \(D_-\Subset D_0\Subset D\Subset Y\).
    The blue points \(z_j\) and their surrounding bump supports lie in \(Y\setminus\overline D\).
    The highlighted simplex \(K\) intersects \(D_-\) and is contained in \(D_0\).
    The dashed polygonal boundary is an abstract representation of \(\partial\omega_R(K)\); it crosses \(\partial D_0\) while remaining strictly inside \(D\).
    The underlying simplicial mesh is deliberately omitted.
    All coefficient profiles satisfy \(a_j\equiv1\) in \(D\).}
    \label{fig:nested-cell-geometry}
\end{figure}

We use the standard periodic cell corrector from scalar periodic homogenization; see, for example, \cite[Theorems~2.3 and~2.6]{Allaire1992} and \cite[Section~4.5.4]{Evans2010}.
\begin{definition}[Periodic cell corrector]
\label{def:corrector}
For a smooth, positive, \(Y\)-periodic scalar coefficient \(a\) and an index \(\ell\in\{1,\dots,d\}\), the \(\ell\)-th periodic cell corrector is the class \(\chi_{a,\ell}\in H^1_{\Per}(Y)/\mathbb R\) determined by
\begin{equation}
    \int_Y a(y) \bigl(e_\ell+\nabla\chi_{a,\ell}(y)\bigr) \cdot\nabla\varphi(y)\di y = 0
    \qquad \bigl(\varphi\in H^1_{\Per}(Y)\bigr).
    \label{eq:cell-problem}
\end{equation}
\end{definition}
Existence and uniqueness modulo constants follow directly from Lax--Milgram on the mean-zero periodic Sobolev space.

\subsection{Exterior periodic dipoles}

Let \(G\) be the mean-zero periodic Green distribution:
\[
    -\Delta G=\delta_0-1 \quad\text{on }\Torus.
\]
Equivalently,
\[
    G(y) = \sum_{n\in\mathbb Z^d\setminus\{0\}} \frac{e^{2\pi\mathrm{i} n\cdot y}}{4\pi^2\lvert n\rvert^2},
\]
in the sense of distributions.
In particular, \(G\) is smooth away from the origin and satisfies \(\Delta G=1\) there.
Its first derivatives are harmonic away from the origin.

The next lemma provides the independent harmonic fields used in the finite-dimensional approximation argument.
The source terms are supported outside \(D\), so the resulting fields are harmonic inside \(D\); nevertheless their gradients remain linearly independent on every nonempty open subset of \(D\).
This independence will later force a positive distance from any \(q\)-dimensional local trial space.

\begin{lemma}[Independent fields generated outside the core]
\label{lem:dipoles}
For every \(q\ge1\), there exist nonnegative functions
\[
    b_1,\dots,b_q \in C_{\mathrm c}^\infty(Y\setminus\overline D),
\]
such that the mean-zero periodic solutions \(v_j\) of
\begin{equation}
    \int_Y\nabla v_j\cdot\nabla\varphi\di y = -\int_Y b_j e_1\cdot\nabla\varphi\di y \qquad \bigl(\varphi\in H^1_{\Per}(Y)\bigr),
    \label{eq:dipole-weak-problem}
\end{equation}
have the following property: the \(q+1\) fields
\begin{equation}
    e_1,\nabla v_1,\dots,\nabla v_q,
    \label{eq:independent-fields}
\end{equation}
are linearly independent in \(L^2(O;\mathbb R^d)\) for every nonempty open set \(O\subset D\).
\end{lemma}

\begin{proof}
We first construct independent fields with point singularities and then replace the point sources by smooth nonnegative bumps.

Choose distinct points \(z_1,\dots,z_q\in Y\setminus\overline D\) and define
\[
    p_j(y) \coloneqq -\partial_{(z_j)_1}G(y-z_j) = \partial_{y_1}G(y-z_j).
\]
Suppose that
\begin{equation}
    c_0e_1+\sum_{j=1}^q c_j\nabla p_j=0,
    \label{eq:singular-field-relation}
\end{equation}
on a nonempty open subset of \(D\).
Set \(M\coloneqq\Torus\setminus\{z_1,\dots,z_q\}\).
Every component of the left-hand side is harmonic, hence real analytic in local coordinates, on \(M\); see, for instance, \cite[Theorem~1.28]{Axler2001}.
The manifold \(M\) is connected, and this is where \(d\ge2\) is used.
The identity theorem for real-analytic functions \cite[Theorem~1.27]{Axler2001}, applied in overlapping coordinate balls along paths in \(M\), therefore propagates the vanishing in \cref{eq:singular-field-relation} from the assumed nonempty open subset of \(D\) to all of \(M\).
Near \(z_j\), all terms other than \(c_j\nabla p_j\) are smooth, whereas
\[
    -\Delta p_j=\partial_{y_1}\delta_{z_j}.
\]
More explicitly, in a coordinate ball about \(z_j\), the periodic Green function is the Euclidean fundamental solution of \(-\Delta\) plus a smooth function.
Hence \(\nabla p_j\) contains a nonzero second derivative of that fundamental solution, of order \(\lvert y-z_j\rvert^{-d}\), and has no smooth extension across \(z_j\).
\cref{eq:singular-field-relation}, whose other terms are smooth near \(z_j\), can therefore hold only if \(c_j=0\).
This argument applies to every \(j\), after which \(c_0=0\).

\proofheading{Mollification.}
Let \(\eta_\delta\) be a nonnegative, unit-mass periodic mollifier supported in the torus ball of radius \(\delta\) about the origin, and set \(b_{j,\delta}(y)\coloneqq\eta_\delta(y-z_j)\).
Choose \(\delta>0\) smaller than the torus distance from every \(z_j\) to \(\overline D\) and smaller than the radius of a periodic coordinate neighborhood about every \(z_j\).
Then each support is contained in \(Y\setminus\overline D\) and is represented by one ordinary coordinate ball, so no ambiguity arises from the chosen fundamental cell.
The solution of \cref{eq:dipole-weak-problem} with \(b_j=b_{j,\delta}\) is, up to an irrelevant additive constant,
\[
    v_{j,\delta} = (\partial_{y_1}G)*b_{j,\delta}.
\]
Since \(z_j\) remains a positive distance from \(\overline D\), for every fixed \(\ell\in\mathbb N_0\),
\[
    \nabla v_{j,\delta} \longrightarrow \nabla p_j
    \quad\text{in }C^\ell(\overline D;\mathbb R^d)
    \quad\text{as }\delta\to0.
\]
The Gram determinant of the finite family \((e_1,\nabla v_{1,\delta},\dots,\nabla v_{q,\delta})\) in \(L^2(D;\mathbb R^d)\) is therefore positive for all sufficiently small \(\delta\).
Fix such a \(\delta\) and write \(b_j\coloneqq b_{j,\delta}\).

The functions \(v_j\) are harmonic on \(D\).
Any linear relation among the fields in \cref{eq:independent-fields} on a nonempty open \(O\subset D\) propagates to \(D\) by harmonic analyticity and is trivial by the preceding Gram-determinant argument.
\end{proof}

\subsection{Perturbation to scalar coefficients}

The exterior bump functions in \cref{lem:dipoles} do not yet define coefficients for the cell problem.
The next lemma connects them to the periodic cell problem.
For a small scalar perturbation \(a_t=1+tb\), the first cell-corrector gradient has first variation \(t\nabla v\), with a quadratic remainder.
Because the perturbation vanishes in the core \(D\), the same \(\bigO(t^2)\) control is available uniformly inside smaller cores, which is what allows the independent dipole fields to become actual corrector fields.

\begin{lemma}[First variation of the corrector]
\label{lem:perturbation}
Let \(b\in C^\infty_{\Per}(Y)\) be nonnegative, set \(a_t\coloneqq1+tb\), and let \(\chi_t\) be the mean-zero representative of \(\chi_{a_t,1}\).
If \(v\) is the mean-zero representative of the solution of \cref{eq:dipole-weak-problem} with \(b_j\) replaced by \(b\), then there is a constant \(C_b>0\), depending only on \(\norm{b}_{L^\infty(Y)}\) and \(\norm{b}_{L^2(Y)}\), such that, for every \(t>0\),
\begin{equation}
    \norm{\nabla\chi_t-t\nabla v}_{L^2(Y)} \le C_b t^2.
    \label{eq:corrector-L2-expansion}
\end{equation}
If \(b=0\) on \(D\), then, for every \(D'\Subset D\), there is a constant \(C_{b,D'}>0\) such that
\begin{equation}
    \norm{\nabla\chi_t-t\nabla v}_{C^0(\overline{D'})} \le C_{b,D'}t^2.
    \label{eq:corrector-C0-expansion}
\end{equation}
Here \(C_{b,D'}\) depends only on \(d,D,D'\), \(\norm{b}_{L^\infty(Y)}\), and \(\norm{b}_{L^2(Y)}\), and is independent of \(t\).
\end{lemma}

\begin{proof}
From \cref{def:corrector}, we have
\begin{equation}
    \int_Y\nabla\chi_t\cdot\nabla\varphi\di y = -t\int_Y b(e_1+\nabla\chi_t)\cdot\nabla\varphi\di y.
    \label{eq:perturbed-cell-equation}
\end{equation}
Testing the original weighted cell problem with \(\chi_t\), using \(\int_Y\partial_1\chi_t\di y=0\), and applying Cauchy--Schwarz yield
\[
    \norm{\nabla\chi_t}_{L^2(Y)}\le t\norm{b}_{L^2(Y)}.
\]
Subtracting \(t\) times the equation for \(v\) from \cref{eq:perturbed-cell-equation} gives
\[
    \int_Y\nabla(\chi_t-tv)\cdot\nabla\varphi\di y = -t\int_Y b\,\nabla\chi_t\cdot\nabla\varphi\di y.
\]
Testing with \(\chi_t-tv\) proves \cref{eq:corrector-L2-expansion}.
For example, one may take
\[
    C_b=1+\norm{b}_{L^\infty(Y)}\norm{b}_{L^2(Y)}.
\]
If \(b=0\) on \(D\), then \(w_t\coloneqq\chi_t-tv\) is harmonic there.
The chosen representatives make \(w_t\) mean zero on \(Y\).
Choose \(D'\Subset D''\Subset D\).
The local derivative estimate for harmonic functions \cite[Theorem~7, p.~29]{Evans2010}, applied on balls with radii controlled by \(\Dist(D',\mathbb R^d\setminus D'')\), and the periodic Poincar\'e inequality give, with \(C_{d,D',D''}\) depending only on the displayed sets and the dimension,
\[
    \norm{\nabla w_t}_{L^\infty(D')} \le C_{d,D',D''}\norm{w_t}_{L^2(D'')} \le C_{d,D',D''}\norm{\nabla w_t}_{L^2(Y)}.
\]
Combining this with \cref{eq:corrector-L2-expansion} proves \cref{eq:corrector-C0-expansion}.
\end{proof}

We now combine the exterior fields with the perturbation lemma.
The proposition produces finitely many periodic scalar coefficients that coincide on \(D\), while their corrector fields in the first coordinate direction remain uniformly separated from every \(q\)-dimensional subspace on every uniformly shape-regular simplex in \(D_0\).
The uniformity over this family of simplices is needed because the coarse mesh is arbitrary.

\begin{proposition}[Finite corrector family with common local coefficient data]
\label{prop:finite-corrector-family}
Fix \(q\ge1\), \(0<h_-\le h_+\), and \(\gamma\ge1\).
Let \(\mathscr S\) be any family of closed \(d\)-simplices \(S\subset\overline D_0\) satisfying
\begin{equation}
    h_-\le\Diam(S)\le h_+, \qquad \frac{\Diam(S)}{r_S}\le\gamma,
    \label{eq:family-simplex-regularity}
\end{equation}
where \(r_S\) is the inradius of \(S\).
There exist smooth periodic scalar coefficients
\begin{equation}
    a_0,\dots,a_q\in C^\infty_{\Per}(Y), \qquad 1\le a_j\le\rho, \qquad a_j=1\ \text{in }D \qquad (0\le j\le q),
    \label{eq:family-coefficients}
\end{equation}
and a constant \(c_{\mathrm{app}}>0\) such that
\begin{equation}
    \sum_{j=0}^q \Dist_{L^2(S;\mathbb R^d)}(g_j,Z)^2 \ge c_{\mathrm{app}},
    \label{eq:family-approximation-lower-bound}
\end{equation}
for every \(S\in\mathscr S\) and every subspace \(Z\subset L^2(S;\mathbb R^d)\) with \(\dim Z\le q\), where
\[
    g_j\coloneqq e_1+\nabla\chi_{a_j,1} \qquad (0\le j\le q).
\]
The profiles \(a_0,\dots,a_q\) and the constant \(c_{\mathrm{app}}\) may depend only on \(d,q,\rho,D,D_0,h_-,h_+\), and \(\gamma\); in particular, they are independent of the family \(\mathscr S\), the individual simplex \(S\), and the subspace \(Z\).
\end{proposition}

\begin{proof}
Choose \(b_1,\dots,b_q\) and \(v_1,\dots,v_q\) as in \cref{lem:dipoles}.
Consider the ordered vertex tuples of all simplices in \(\overline D_0\) that satisfy \cref{eq:family-simplex-regularity}, not only those in \(\mathscr S\).
These tuples lie in the compact set \(\overline D_0^{\,d+1}\).
Their closure is compact, and the bound \(r_S\ge h_-/\gamma\) prevents a limiting simplex from degenerating.
For every simplex represented by a tuple in this compact closure, the Gram matrix of
\[
    e_1,\nabla v_1,\dots,\nabla v_q
\]
in \(L^2(S;\mathbb R^d)\) is positive definite.
Indeed, the interior of every limiting simplex is a nonempty open subset of \(D\), so \cref{lem:dipoles} applies.
To see the dependence on \(S\) explicitly, pull \(S\) back from a fixed reference simplex by its affine vertex map.
Each Gram entry is then an integral over the reference simplex of a jointly continuous integrand, multiplied by the absolute Jacobian determinant of the affine map.
Since the limiting simplices are nondegenerate, these integrals depend continuously on the ordered vertices.
Compactness gives a uniform positive lower bound for its smallest eigenvalue.

Denote the preceding uniform Gram lower bound by \(\lambda_0>0\), and set
\[
    B_\infty\coloneqq\max_{1\le j\le q}\norm{b_j}_{L^\infty(Y)}.
\]
Set \(a_0\coloneqq1\) and \(a_j\coloneqq1+tb_j\) for \(1\le j\le q\).
We now choose \emph{one} \(t>0\) satisfying
\[
    tB_\infty\le\rho-1
\]
and small enough for the perturbation argument below.
Thus, \(t\) is allowed to depend on \(\rho\), on the fixed profiles \(b_j\), and on the geometric data entering \(\lambda_0\), but it is independent of \(\mathscr S\) and \(S\in\mathscr S\).
By \cref{lem:perturbation}, uniformly on \(\overline D_0\),
\[
    g_0=e_1, \qquad \frac{g_j-g_0}{t} = \nabla v_j+\bigO(t) \qquad (1\le j\le q).
\]
Consequently, after decreasing this fixed \(t\) if necessary, the Gram matrix of the transformed family
\[
    g_0,\frac{g_1-g_0}{t},\dots,\frac{g_q-g_0}{t}
\]
has smallest eigenvalue at least \(\lambda_0/2\) for every \(S\in\mathscr S\).
Let \(\widehat e_0,\dots,\widehat e_q\) denote the standard basis of \(\mathbb R^{q+1}\), and let \(A_t\in\mathbb R^{(q+1)\times(q+1)}\) be the matrix whose first column is \(\widehat e_0\) and whose \(j\)-th column is \(\widehat e_0+t\widehat e_j\) for \(1\le j\le q\).
The column family \((g_0,\dots,g_q)\) is obtained from the transformed family by right multiplication with \(A_t\).
Since the chosen \(t\) is positive, \(A_t\) is invertible, and hence
\[
    \lambda_{\min}\bigl(\operatorname{Gram}_S(g_0,\dots,g_q)\bigr)
    \ge \frac{\lambda_0}{2}\,\sigma_{\min}(A_t)^2
    \eqqcolon \lambda_*(t)>0
    \qquad \bigl(S\in\mathscr S\bigr).
\]
This lower bound is uniform in \(S\), not in \(t\).
Its dependence on the fixed \(t\), and therefore on the contrast bound \(\rho\), is retained; indeed, it must deteriorate as \(t\downarrow0\), when \(g_0,\dots,g_q\) coalesce.

To relate this fact to \cref{eq:family-approximation-lower-bound}, define
\[
    B_S\colon\mathbb R^{q+1}\longrightarrow L^2(S;\mathbb R^d), \qquad B_S c\coloneqq\sum_{j=0}^q c_j g_j, \qquad c=(c_0,\dots,c_q).
\]
Every subspace \(Z\) of dimension at most \(q\) is closed.
If \(P_Z\) is its orthogonal projection, then
\[
    \sum_{j=0}^q\Dist(g_j,Z)^2 = \norm{(I-P_Z)B_S}_{\mathrm{HS}}^2.
\]
Here \(\norm{\cdot}_{\mathrm{HS}}\) denotes the Hilbert--Schmidt norm for operators from Euclidean \(\mathbb R^{q+1}\) into \(L^2(S;\mathbb R^d)\).
The best rank-\(q\) approximation inequality for singular values gives
\[
    \norm{(I-P_Z)B_S}_{\mathrm{HS}}^2 \ge \sigma_{q+1}(B_S)^2,
\]
where the singular values are ordered nonincreasingly; moreover, \(\sigma_{q+1}(B_S)^2\) is the smallest eigenvalue of the Gram matrix.
Therefore \cref{eq:family-approximation-lower-bound} holds with \(c_{\mathrm{app}}=\lambda_*(t)\).
As asserted in the proposition, this constant may depend on \(\rho\) through the fixed choice of \(t\), but it is independent of \(\mathscr S\), \(S\), and \(Z\).
\end{proof}

\section{A positive density of elements with patches contained in the core}

The corrector family is used on coarse elements whose \(R\)-layer coefficient-information patches lie in a region where all coefficients in the family coincide.
This section proves that, after choosing the period \(\varepsilon_H=LH\) with \(L\) fixed, every sufficiently fine quasi-uniform mesh contains \(\Theta(H^{-d})\) such elements.
It also records the compactness of their rescaled shapes, which is the geometric input needed to apply the finite-family lower bound uniformly.
The lattice in \cref{fig:safe-element-density} is the periodic array of copies of \(D_-\), not an assumption that the mesh is structured.

\begin{figure}[H]
    \centering
    \includegraphics[width=0.90\textwidth]{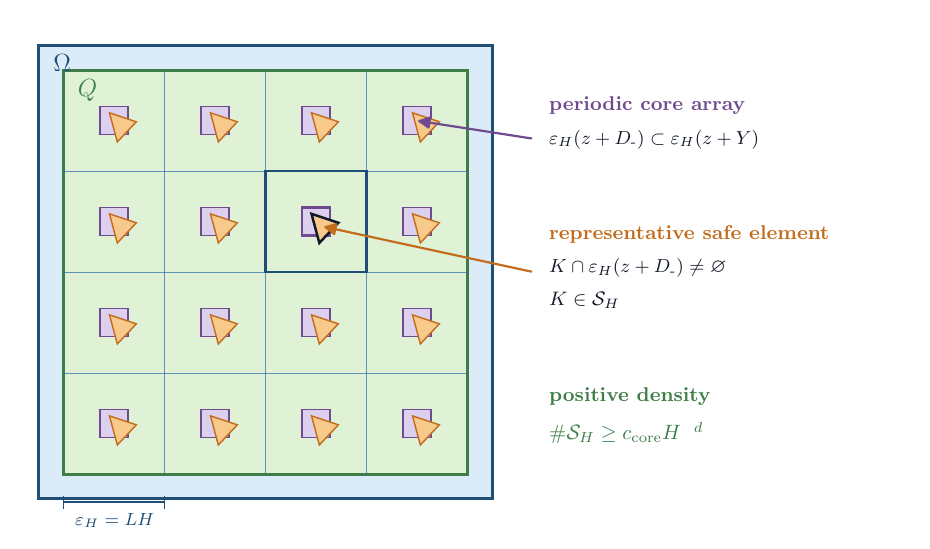}
    \caption{Schematic of the positive-density patch-containment argument.
    The coefficient period is \(\varepsilon_H=LH\); the thin blue lines delimit periodic cells \(\varepsilon_H(z+Y)\), and the purple squares are the corresponding copies of \(D_-\) inside the interior cube \(Q\).
    If \(L\) is chosen large enough, every mesh element meeting one of these copies lies in the corresponding copy of \(D_0\), and its full \(R\)-layer patch lies in the corresponding copy of \(D\).
    Since the union of the \(D_-\) copies occupies a fixed fraction of \(Q\), quasi-uniform volume counting gives \(\#\mathcal S_H\ge c_{\mathrm{core}}H^{-d}\).
    The mesh itself is deliberately not drawn and need not be Cartesian or periodic; the orange simplices represent elements whose patches are contained in the core.}
    \label{fig:safe-element-density}
\end{figure}

The next lemma fixes the physical period relative to \(H\) and counts the elements whose coefficient-information patches remain inside one copy of \(D\).
It also records the rescaled shape bounds needed to apply the finite-family proposition on each such element.

\begin{lemma}[Patch containment on a positive fraction of elements]
\label{lem:safe}
Fix \(R\in\mathbb N_0\) and the cubes in \cref{eq:nested-cubes}.
There exist \(L>1\), a cube \(Q\Subset\Omega\), and constants \(c_{\mathrm{core}},H_0>0\) such that, with
\[
    \varepsilon_H\coloneqq LH,
\]
the set
\[
    \mathcal S_H \coloneqq \left\{ K\in\mathcal T_H \ \middle|\
    \begin{aligned}
        &K\subset Q,\ \text{and for some }z\in\mathbb Z^d,\\
        &K\subset\varepsilon_H(z+D_0), \quad \omega_{R}(K)\subset\varepsilon_H(z+D)
    \end{aligned}
    \right\}
\]
satisfies
\begin{equation}
    \#\mathcal S_H\ge c_{\mathrm{core}}H^{-d}
    \label{eq:safe-element-count}
\end{equation}
whenever \(H\in\mathcal H\) and \(H<H_0\).
Moreover, the collection of rescaled simplices over all such \(H\) and \(K\),
\[
    S_K \coloneqq \varepsilon_H^{-1}K-z, \qquad K\in\mathcal S_H,
\]
lies in \(\overline D_0\), is uniformly shape-regular, and satisfies
\begin{equation}
    \frac{c_{\mathrm{qu}}}{L}\le\Diam(S_K)\le\frac1L, \qquad \frac{\Diam(S_K)}{r_{S_K}}\le\gamma_{\mathrm{sh}}.
    \label{eq:rescaled-safe-regularity}
\end{equation}
The choices of \(L,Q,c_{\mathrm{core}}\), and \(H_0\) may depend only on \(d,R,\Omega\) and the fixed cubes \(D_-,D_0,D\); in particular, they are independent of \(H\), the index set \(\mathcal H\), and the particular mesh family.
\end{lemma}

\begin{proof}
\proofheading{From graph layers to Euclidean distance.}
If \(T_0=K,T_1,\dots,T_j\) is a chain of face-neighboring elements with \(j\le R\), then any point of \(T_j\) is at Euclidean distance at most \((R+1)H\) from \(K\).
Hence
\begin{equation}
    \omega_{R}(K) \subset \bigl\{ x\in\overline\Omega \bigm| \Dist(x,K)\le C_{\mathrm{vis}}H \bigr\}, \qquad C_{\mathrm{vis}}\coloneqq R+1.
    \label{eq:patch-euclidean-bound}
\end{equation}

\proofheading{Choice of the physical period.}
Choose axis-parallel cubes \(Q'\Subset Q\Subset\Omega\), and denote the side length of \(Q'\) by \(\ell_Q\).
Set
\[
    \delta_{-0} \coloneqq \Dist(\overline D_-,\mathbb R^d\setminus D_0)>0, \qquad \delta_{0D} \coloneqq \Dist(\overline D_0,\mathbb R^d\setminus D)>0.
\]
Fix \(L>1\) so that
\begin{equation}
    L\delta_{-0}>2, \qquad L\delta_{0D}>2C_{\mathrm{vis}}.
    \label{eq:period-separation-choice}
\end{equation}
Finally choose \(H_0>0\) small enough that
\begin{equation}
    LH_0<\frac{\ell_Q}{4}, \qquad H_0<\frac12\Dist(\overline Q',\Omega\setminus Q).
    \label{eq:safe-H0-choice}
\end{equation}

\proofheading{Full-cell count and a volume lower bound.}
Define the periodically repeated inner cores
\[
    E_H \coloneqq Q' \cap \bigcup_{z\in\mathbb Z^d} \varepsilon_H(z+D_-).
\]
In each coordinate direction, the interval defining \(Q'\) contains at least
\[
    \frac{\ell_Q}{\varepsilon_H}-2 \ge \frac{\ell_Q}{2\varepsilon_H}
\]
closed intervals of the form \(\varepsilon_H(z_j+[-1/2,1/2])\); the last inequality follows from \(\varepsilon_H<\ell_Q/4\).
Hence at least
\[
    N_{\mathrm{cell}}(H) \ge \left(\frac{\ell_Q}{2\varepsilon_H}\right)^d
\]
closed cells \(\varepsilon_H(z+\overline Y)\) are contained in \(Q'\).
Their copies of \(\varepsilon_H(z+D_-)\) are pairwise disjoint and lie in \(E_H\).
It follows that
\begin{equation}
    \lvert E_H\rvert \ge N_{\mathrm{cell}}(H)\varepsilon_H^d\lvert D_-\rvert \ge 2^{-d}\ell_Q^d\lvert D_-\rvert \eqqcolon c_E>0.
    \label{eq:inner-core-volume-bound}
\end{equation}

\proofheading{Patch containment for elements meeting \(E_H\).}
Let \(K\in\mathcal T_H\) satisfy \(\lvert K\cap E_H\rvert>0\).
By the definition of \(E_H\), there exist \(z\in\mathbb Z^d\) and a point
\[
    x_0\in K\cap Q'\cap\varepsilon_H(z+D_-).
\]
For every \(x\in K\),
\[
    \lvert x-x_0\rvert \le \Diam(K) \le H < \varepsilon_H\delta_{-0},
\]
where the strict inequality follows from \cref{eq:period-separation-choice}.
The definition of \(\delta_{-0}\) therefore gives
\begin{equation}
    K\subset\varepsilon_H(z+D_0).
    \label{eq:safe-K-in-D0}
\end{equation}
Moreover, every point of \(K\) is within distance \(H\) of \(x_0\).
Together with \cref{eq:safe-H0-choice}, this implies \(K\subset Q\).

If \(x\in\omega_{R}(K)\), then \cref{eq:patch-euclidean-bound,eq:safe-K-in-D0} give
\[
    \Dist\bigl( x,\varepsilon_H(z+\overline D_0) \bigr) \le C_{\mathrm{vis}}H < \varepsilon_H\delta_{0D}.
\]
By the definition of \(\delta_{0D}\), this proves
\[
    \omega_{R}(K)\subset\varepsilon_H(z+D).
\]
Thus \(K\in\mathcal S_H\).

\proofheading{Conversion of volume to element count.}
The mesh elements cover \(E_H\) up to their measure-zero interfaces.
Since every simplex of diameter at most \(H\) has volume at most \(C_d H^d\), where \(C_d\) depends only on \(d\), \cref{eq:inner-core-volume-bound} yields
\[
    c_E \le \lvert E_H\rvert \le \sum_{\substack{K\in\mathcal T_H\\ \lvert K\cap E_H\rvert>0}} \lvert K\rvert \le C_d H^d\#\mathcal S_H.
\]
Consequently \cref{eq:safe-element-count} holds with \(c_{\mathrm{core}}\coloneqq c_E/C_d\).

\proofheading{Regularity of the rescaled elements.}
Scaling preserves shape regularity, and every \(S_K\) lies in \(\overline D_0\).
The definition \(\varepsilon_H=LH\) and \cref{eq:mesh-regularity} therefore give \cref{eq:rescaled-safe-regularity}.
\end{proof}

\section{Realizing the corrector fields by smooth right-hand sides}
\label{sec:homogenization}

It remains to realize the cell fields constructed in \cref{sec:finite-family} as gradients of exact solutions of the original boundary-value problem.
This section constructs smooth compactly supported right-hand sides whose homogenized solutions are affine on the cube \(Q\), so that strong corrector convergence yields the corresponding fields \(g_s\) on the elements in \(\mathcal S_H\).

The first result in this section isolates the only external homogenization input.
It is stated in the form needed below.
A smooth homogenized solution with a smooth compactly supported right-hand side admits the standard strong first-order corrector expansion.

\begin{proposition}[Periodic corrector convergence used here]
\label{prop:allaire-corrector}
Let \(a\in C^\infty_{\Per}(Y)\) satisfy \(1\le a\le\rho\), let \(A^{\mathrm{hom}}\) be its homogenized matrix, and let \(\chi_{a,\ell}\) be the mean-zero cell correctors from \cref{eq:cell-problem}.
For \(U\in C_{\mathrm c}^\infty(\Omega)\), set
\[
    f\coloneqq-\nabla\cdot(A^{\mathrm{hom}}\nabla U).
\]
If \(\varepsilon\to0\), \(\kappa_\varepsilon(x)\coloneqq a(x/\varepsilon)\), and \(u_\varepsilon\coloneqq u_{\kappa_\varepsilon,f}\), then
\begin{equation}
    \norm{ \nabla u_\varepsilon-\nabla U-\sum_{\ell=1}^d\nabla_y\chi_{a,\ell}(x/\varepsilon)\,\partial_\ell U }_{L^2(\Omega)} \longrightarrow0.
    \label{eq:allaire-corrector-input}
\end{equation}
\end{proposition}

\begin{proof}
This is the only external homogenization result used in the paper.
It is the specialization of Allaire's two-scale homogenization theorem and strong gradient-corrector theorem \cite[Theorems~2.3 and~2.6, pp.~1494 and~1496]{Allaire1992} to the matrix \(A(x,y)=a(y)I\).
The coefficient is smooth, \(Y\)-periodic, uniformly elliptic, and independent of \(x\), so it satisfies the admissibility hypotheses of \cite[Theorem~2.3]{Allaire1992}.
The homogenized solution is \(U\), because \(f=-\nabla\cdot(A^{\mathrm{hom}}\nabla U)\) and \(U\in H_0^1(\Omega)\).
In Allaire's notation, the microscopic part of the two-scale limit is
\[
    u_1(x,y)=\sum_{\ell=1}^d\chi_{a,\ell}(y)\,\partial_\ell U(x).
\]
Since \(a\) and \(U\) are smooth, elliptic regularity for the periodic cell problems makes \(\nabla_y u_1\) smooth in both variables.
It is therefore an admissible test function in the strong corrector theorem; see the first paragraph of the proof of \cite[Theorem~2.6, p.~1496]{Allaire1992}.
Translating Allaire's periodic unit cell to \(Y=(-1/2,1/2)^d\) does not change the theorem.
Uniqueness of the homogenized solution and of the mean-zero cell correctors identifies every subsequential limit, so \cref{eq:allaire-corrector-input} holds for the whole family \(\varepsilon\to0\).
\end{proof}

The next lemma chooses one common macroscopic function \(U\) for the whole finite family.
Because \(U\) is affine on \(Q\), the first-order expansion on \(Q\) reduces to the first cell-corrector field \(g_j\); this makes the local finite-family lower bound applicable to exact solutions.

\begin{lemma}[Strong corrector convergence for smooth right-hand sides]
\label{lem:homogenization}
Let \(a_0,\dots,a_q\) be the coefficients from \cref{prop:finite-corrector-family}, and let \(A_j^{\mathrm{hom}}\) be the homogenized matrix associated with \(a_j\).
Choose \(U\in C_{\mathrm c}^\infty(\Omega)\) such that
\begin{equation}
    U(x)=x_1 \quad \text{on a neighborhood of }\overline Q,
    \label{eq:affine-homogenized-solution}
\end{equation}
where \(Q\) is supplied by \cref{lem:safe}, and define
\begin{equation}
    f_j \coloneqq -\nabla\cdot(A_j^{\mathrm{hom}}\nabla U) \in C_{\mathrm c}^\infty(\Omega)
    \qquad (0\le j\le q).
    \label{eq:family-source}
\end{equation}
For
\[
    \kappa_{H,j}(x) \coloneqq a_j(x/\varepsilon_H)
\]
and \(u_{H,j}\coloneqq u_{\kappa_{H,j},f_j}\), one has, for every \(j\in\{0,\dots,q\}\),
\begin{equation}
    \norm{ \nabla u_{H,j} - \nabla U - \sum_{\ell=1}^d \nabla_y\chi_{a_j,\ell}(x/\varepsilon_H) \,\partial_\ell U }_{L^2(\Omega)} \longrightarrow0,
    \label{eq:strong-corrector-convergence}
\end{equation}
as \(H\to0\) through \(\mathcal H\).
In particular, recalling from \cref{prop:finite-corrector-family} that
\[
    g_j\coloneqq e_1+\nabla\chi_{a_j,1},
\]
it follows that
\begin{equation}
    \norm{\nabla u_{H,j}-g_j(x/\varepsilon_H)}_{L^2(Q)} \longrightarrow0.
    \label{eq:corrector-convergence-on-Q}
\end{equation}
There is a constant \(C_U\), independent of \(H\) and of \(j\in\{0,\dots,q\}\), such that
\begin{equation}
    0<\norm{f_j}_{L^2(\Omega)}\le C_U
    \qquad (0\le j\le q).
    \label{eq:source-norm-bound}
\end{equation}
The constant \(C_U\) may be chosen to depend only on \(d,\rho\), and \(\norm{D^2U}_{L^2(\Omega)}\).
Only this uniform upper bound, rather than a uniform positive lower bound for \(\norm{f_j}_{L^2(\Omega)}\), is used below.
\end{lemma}

\begin{proof}
Here \(\nabla_y\) denotes differentiation with respect to the periodic cell variable \(y\).
Fix \(j\in\{0,\dots,q\}\).
The periodic cell formula and its variational characterization \cite[Section~4.5.4 and Problem~18 in Section~4.8]{Evans2010} give, for every \(\xi\in\mathbb R^d\),
\begin{equation}
    \xi\cdot A_j^{\mathrm{hom}}\xi
    =
    \min_{\phi\in H^1_{\Per}(Y)/\mathbb R}
    \int_Y a_j(y)\lvert \xi+\nabla\phi(y)\rvert^2\di y.
    \label{eq:homogenized-variational-formula}
\end{equation}
Since \(\int_Y\nabla\phi=0\) for periodic \(\phi\), \cref{eq:homogenized-variational-formula} and \(1\le a_j\le\rho\) imply
\[
    I\preceq A_j^{\mathrm{hom}}\preceq\rho I.
\]
Here \(A\preceq B\) means that \(B-A\) is positive semidefinite.
Consequently, \cref{eq:family-source} is smooth and compactly supported, and one may take
\[
    C_U=\sqrt d\,\rho\norm{D^2U}_{L^2(\Omega)}.
\]
It is nonzero: otherwise \(U\in H_0^1(\Omega)\) would be \(A_j^{\mathrm{hom}}\)-harmonic, so testing its equation with \(U\) would give \(U=0\), contradicting \cref{eq:affine-homogenized-solution}.

\cref{prop:allaire-corrector}, applied to \(a=a_j\) and \(\varepsilon=\varepsilon_H\), gives \cref{eq:strong-corrector-convergence} as \(H\to0\) through \(\mathcal H\).
Since \(\nabla U=e_1\) on \(Q\), only the first cell corrector remains there, which gives \cref{eq:corrector-convergence-on-Q}.
The family is finite, so the convergence and the source bounds are uniform over \(j=0,\dots,q\).
\end{proof}

\section{Proof of the order-one lower bound}
\label{sec:main-proof}

We now combine the preceding results.
The proof uses four results: the common local space in \cref{lem:restriction}, the finite-family approximation bound in \cref{prop:finite-corrector-family}, the element count in \cref{lem:safe}, and the strong corrector convergence in \cref{lem:homogenization}.
The proposition is stated separately from the proof of \cref{thm:main} because it records the stronger finite-family conclusion: the profiles and sources are fixed before \(H\) and before the rule, while only the selected index is chosen after the rule is known.
\begin{proposition}[Finite-family realization]
\label{prop:finite-family-realization}
Under the assumptions of \cref{thm:main}, let \(q\) be defined by \cref{eq:q-and-R}.
There exist \(L>1\), constants \(c_*,H_*>0\), coefficient profiles
\[
    a_0,\dots,a_q\in C^\infty_{\Per}(Y), \qquad 1\le a_j\le\rho \qquad (0\le j\le q),
\]
and nonzero right-hand sides
\[
    f_0,\dots,f_q\in C_{\mathrm c}^\infty(\Omega)
\]
with the following property.
All these objects can be chosen independently of \(H\) and of the selection rule.
For every \(H\in\mathcal H\) with \(0<H<H_*\) and every fixed-visibility rule \(\mathscr M\) with parameters \((H,m,k,C_{\mathrm{dim}},C_{\mathrm{loc}})\), there exists \(s\in\{0,\dots,q\}\) such that, after setting
\begin{equation}
    \kappa_{H,s}(x)\coloneqq a_s\bigl(x/(LH)\bigr),
    \label{eq:finite-family-coefficient}
\end{equation}
and
\[
    V_{H,s}\coloneqq V_H^\mathscr M(\kappa_{H,s}), \qquad
    u_{H,s}\coloneqq u_{\kappa_{H,s},f_s},
\]
and letting \(u_{\kappa_{H,s},f_s,H}\in V_{H,s}\) denote the corresponding Galerkin solution, one has
\begin{equation}
    \frac{
        \norm{
            u_{H,s}
            -u_{\kappa_{H,s},f_s,H}
        }_{a_{\kappa_{H,s}}}
    }{
        \norm{f_s}_{L^2(\Omega)}
    }
    \ge c_*.
    \label{eq:finite-family-lower-bound}
\end{equation}
The number \(L\), the profiles \(a_j\), the right-hand sides \(f_j\), and the constants \(c_*,H_*\) may depend only on
\[
    d,\rho,m,k,C_{\mathrm{loc}},\Omega,\gamma_{\mathrm{sh}},c_{\mathrm{qu}}.
\]
In particular, these objects are independent of \(C_{\mathrm{dim}}\), \(H\), the selection rule, and the particular mesh family satisfying the stated regularity bounds.
\end{proposition}

\begin{proof}[Proof of \cref{prop:finite-family-realization}]
The case \(q=0\) was settled before the corrector-family construction, so assume \(q\ge1\).
Apply \cref{lem:safe} with \(R=2m+k\), and let \(\mathscr S_*\) be the family of all closed simplices \(S\subset\overline D_0\) such that
\[
    \frac{c_{\mathrm{qu}}}{L} \le \Diam(S) \le \frac1L, \qquad \frac{\Diam(S)}{r_S} \le \gamma_{\mathrm{sh}},
\]
where \(r_S\) is the inradius.
This is a uniformly shape-regular family with diameters bounded above and away from zero, and it depends only on the displayed fixed parameters, not on the mesh family.
Every rescaled simplex \(S_K\) with \(K\in\mathcal S_H\) belongs to \(\mathscr S_*\) by \cref{eq:rescaled-safe-regularity}.
\cref{prop:finite-corrector-family} applied to \(\mathscr S_*\) therefore gives coefficients \(a_0,\dots,a_q\) and a constant \(c_{\mathrm{app}}>0\) that are independent of \(H\) and of the rule.

\proofheading{Step 1: A common local space.}
Fix a sufficiently small \(H\) and a rule \(\mathscr M\).
For \(s\in\{0,\dots,q\}\), let \(\kappa_{H,s}\) be the coefficient defined in \cref{eq:finite-family-coefficient}; each belongs to \(\mathcal K_\rho\).
Denote the corresponding selected space by
\[
    V_{H,s}\coloneqq V_H^\mathscr M(\kappa_{H,s}).
\]
If \(K\in\mathcal S_H\), then, for its associated \(z\in\mathbb Z^d\),
\[
    \omega_R(K) \subset \varepsilon_H(z+D).
\]
Periodicity and \cref{eq:family-coefficients} show that all \(\kappa_{H,s}\) equal one on this patch.
\cref{lem:restriction} therefore shows that the spaces \(Y_{K,s}\coloneqq Y_K^{\kappa_{H,s}}\) are independent of \(s\) and have dimension at most \(q\).
Denote this common space by
\begin{equation}
    Y_K \coloneqq Y_{K,s} \quad (s=0,\dots,q), \qquad \dim Y_K\le q.
    \label{eq:common-safe-restriction}
\end{equation}

The gradient map is well defined and injective on \(H^1(K)/\mathbb R\).
Hence
\[
    Z_K \coloneqq \bigl\{ \nabla w \bigm| [w]\in Y_K \bigr\} \subset L^2(K;\mathbb R^d), \qquad \dim Z_K\le q.
\]

\proofheading{Step 2: A local approximation lower bound on each selected element.}
Under the change of variables \(x=\varepsilon_H(z+y)\), let
\[
    \widehat Z_K \coloneqq \bigl\{ z_H(\varepsilon_H(z+\cdot)) \bigm| z_H\in Z_K \bigr\} \subset L^2(S_K;\mathbb R^d).
\]
The pullback preserves the dimension of the space.
Applying \cref{prop:finite-corrector-family} on \(S_K\), using the periodic identity \(g_s(z+y)=g_s(y)\), and scaling the \(L^2\)-norm gives
\[
    \begin{aligned}
        &\sum_{s=0}^q
        \Dist_{L^2(K;\mathbb R^d)}
        \bigl(g_s(x/\varepsilon_H),Z_K\bigr)^2 \\
        &\qquad =
        \varepsilon_H^d
        \sum_{s=0}^q
        \Dist_{L^2(S_K;\mathbb R^d)}(g_s,\widehat Z_K)^2 \\
        &\qquad \ge c_{\mathrm{app}}\varepsilon_H^d
        = c_{\mathrm{elem}}H^d,
    \end{aligned}
\]
where \(c_{\mathrm{elem}}\coloneqq c_{\mathrm{app}}L^d\).
Summing over the disjoint mesh elements in \(\mathcal S_H\) and using \cref{eq:safe-element-count} yields
\[
    \sum_{K\in\mathcal S_H} \sum_{s=0}^q
    \Dist_{L^2(K;\mathbb R^d)}
    \bigl(g_s(x/\varepsilon_H),Z_K\bigr)^2
    \ge c_1,
    \qquad c_1\coloneqq c_{\mathrm{elem}}c_{\mathrm{core}}>0.
\]

\proofheading{Step 3: Transfer to exact solutions.}
For \(s=0,\dots,q\), let \(f_s\) and \(u_{H,s}\) be supplied by \cref{lem:homogenization}.
As above, \(u_{\kappa_{H,s},f_s,H}\) denotes the Galerkin solution in \(V_{H,s}\).
For any closed subspace \(Z\) of a Hilbert space, the distance function is one-Lipschitz.
Thus, for each \(K\in\mathcal S_H\),
\[
    \begin{aligned}
        \Dist_{L^2(K;\mathbb R^d)}(\nabla u_{H,s},Z_K)
        &\ge
        \Bigl(
            \Dist_{L^2(K;\mathbb R^d)}
            \bigl(g_s(x/\varepsilon_H),Z_K\bigr) \\
        &\qquad\qquad
            -\norm{\nabla u_{H,s}-g_s(x/\varepsilon_H)}_{L^2(K)}
        \Bigr)_+ .
    \end{aligned}
\]
The elementary inequality
\[
    (A-B)_+^2\ge\frac12A^2-B^2 \qquad (A,B\ge0)
\]
and \cref{eq:corrector-convergence-on-Q} now give
\begin{equation}
    \begin{aligned}
        &\sum_{s=0}^q \sum_{K\in\mathcal S_H}
        \Dist_{L^2(K;\mathbb R^d)}
        \bigl(\nabla u_{H,s},Z_K\bigr)^2\\
        &\qquad\ge
        \frac12c_1
        -
        \sum_{s=0}^q
        \norm{\nabla u_{H,s}-g_s(x/\varepsilon_H)}_{L^2(Q)}^2 \\
        &\qquad=
        \frac12c_1-o(1).
    \end{aligned}
    \label{eq:transfer-to-exact-solutions}
\end{equation}
The remainder \(o(1)\) is uniform in \(s\) because the family is finite.
Choose \(H_{\mathrm{corr}}>0\) so that, for \(H\in\mathcal H\) with \(0<H<H_{\mathrm{corr}}\), the corrector-error sum in \cref{eq:transfer-to-exact-solutions} is at most \(c_1/4\).
This threshold depends only on the finite profiles \(a_s\), the fixed function \(U\), \(L\), and \(\Omega\), and is independent of the mesh family, \(H\), and the rule.
There is then an index \(s=s(H,\mathscr M)\) such that
\[
    \sum_{K\in\mathcal S_H}
    \Dist_{L^2(K;\mathbb R^d)}
    \bigl(\nabla u_{H,s},Z_K\bigr)^2
    \ge \frac{c_1}{4(q+1)}.
\]

\proofheading{Step 4: From the local approximation bound to the Galerkin error.}
For every \(v_H\in V_{H,s}\), the gradient of its restriction to \(K\) belongs to \(Z_K\).
Moreover, \(\kappa_{H,s}=1\) on every \(K\in\mathcal S_H\).
Hence
\begin{equation}
    \begin{aligned}
        \norm{u_{H,s}-v_H}_{a_{\kappa_{H,s}}}^2 &\ge \sum_{K\in\mathcal S_H} \norm{\nabla(u_{H,s}-v_H)}_{L^2(K)}^2\\
        &\ge \frac{c_1}{4(q+1)}.
    \end{aligned}
    \label{eq:best-approximation-lower-bound}
\end{equation}
Galerkin orthogonality gives the best-approximation identity
\[
    \norm{u_{H,s}-u_{\kappa_{H,s},f_s,H}}_{a_{\kappa_{H,s}}}
    =
    \inf_{v_H\in V_{H,s}}
    \norm{u_{H,s}-v_H}_{a_{\kappa_{H,s}}}.
\]
Combining this identity with \cref{eq:source-norm-bound,eq:best-approximation-lower-bound} proves
\[
    \frac{
        \norm{u_{H,s}-u_{\kappa_{H,s},f_s,H}}_{a_{\kappa_{H,s}}}
    }{
        \norm{f_s}_{L^2(\Omega)}
    }
    \ge \frac1{C_U} \sqrt{\frac{c_1}{4(q+1)}} \eqqcolon c_*>0.
\]
\proofheading{Step 5: Dependencies and quantifiers.}
All thresholds and constants were chosen from the fixed data and the finite family, never from \(H\) or \(\mathscr M\).
Here and above, \(o(1)\) denotes a quantity tending to zero as \(H\to0\) through \(\mathcal H\).
With \(H_0\) from \cref{lem:safe}, set
\[
    H_*\coloneqq\min\{H_0,H_{\mathrm{corr}}\}.
\]
Taking \(s=s(H,\mathscr M)\) proves \cref{eq:finite-family-lower-bound}.
\end{proof}

\begin{proof}[Proof of \cref{thm:main}]
Fix \(H\in\mathcal H\) with \(0<H<H_*\) and a rule \(\mathscr M\).
\cref{prop:finite-family-realization} supplies an index \(s\), a coefficient \(\kappa_H\coloneqq\kappa_{H,s}\in\mathcal K_\rho\), and a nonzero right-hand side \(f_s\) satisfying \cref{eq:finite-family-lower-bound}.
Because the operator error in \cref{eq:operator-error} is the supremum over all nonzero right-hand sides, \cref{eq:finite-family-lower-bound} implies \cref{eq:main-lower-bound}.
Taking the supremum over \(\kappa\), the infimum over rules, and then the lower limit as \(H\to0\) gives \cref{eq:minimax-lower-bound}.
The dependency statement follows from \cref{prop:finite-family-realization}.
\end{proof}

\section{Consequences for scalable construction and scope}
\label{sec:scope}

The theorem gives a necessary condition for coefficient-uniform convergence of deterministic generalized multiscale finite element space constructions over \(\mathcal K_\rho\).
Such a family cannot keep the support radius, coefficient-information radius, and local multiplicity all fixed while satisfying the anchored consistency condition in \cref{def:rule}.
It must allow at least one of these parameters to grow or use coefficient information from outside the prescribed local patches.
This is a necessary condition on the construction, not a claim that growth of any particular parameter is sufficient and not a runtime lower bound.
Its precise boundary is determined by the order in which the rule, the coefficient, and the trial space are chosen.

\subsection{Exact order of choices}

The stronger finite-family statement in \cref{prop:finite-family-realization} can be read in three steps.
First, the structural parameters are fixed, and the proof chooses one scale factor \(L>1\), a finite family
\[
    \{(a_s,f_s)\}_{s=0}^q
    \subset C^\infty_{\Per}(Y)
    \times \bigl(C_{\mathrm c}^\infty(\Omega)\setminus\{0\}\bigr),
\]
of smooth periodic coefficient profiles \(a_s\) and smooth compactly supported right-hand sides \(f_s\), together with constants \(c_*,H_*>0\).
Second, an arbitrary scale \(H\in\mathcal H\), with \(0<H<H_*\), and an arbitrary fixed-visibility rule \(\mathscr M\) are fixed.
Third, one index \(s=s(H,\mathscr M)\) is selected from the already fixed family.
With
\[
    \kappa_{H,s}(x)=a_s\bigl(x/(LH)\bigr),
    \qquad
    V_{H,s}=V_H^\mathscr M(\kappa_{H,s}),
\]
the corresponding Galerkin solution satisfies
\[
    \frac{
        \norm{u_{\kappa_{H,s},f_s}-u_{\kappa_{H,s},f_s,H}}_{a_{\kappa_{H,s}}}
    }{
        \norm{f_s}_{L^2(\Omega)}
    }
    \ge c_*.
\]
Thus the index \(s\) may depend on the scale and on the rule, but the finite coefficient profiles, the right-hand sides, and the lower-bound constant do not.

The quantities \(L\), \(c_*\), \(H_*\), and the finite family depend only on
\[
    d,\rho,m,k,C_{\mathrm{loc}},\Omega,
    \gamma_{\mathrm{sh}},c_{\mathrm{qu}}.
\]
They are independent of \(H\), the rule, the particular mesh family satisfying the stated regularity bounds, and the global dimension constant \(C_{\mathrm{dim}}\).
Only the local anchor multiplicity \(C_{\mathrm{loc}}\) enters the proof.
The construction applies to every \(\rho>1\), however close to one.
The restriction \(d\ge2\) is used only to construct the independent periodic fields in \cref{lem:dipoles}.

\subsection{Global coefficient information lies outside the model}

For fixed visibility, one rule must act consistently on the entire coefficient class before the coefficient used in the lower bound is selected.
The relevant worst-case quantity is
\[
    \inf_{\mathscr M\ \text{fixed-visibility}}
    \sup_{\kappa\in\mathcal K_\rho}
    \mathcal E_H\bigl(\kappa;V_H^\mathscr M(\kappa)\bigr).
\]
If only support locality is imposed, the whole coefficient may be examined before the trial space is designed.
The corresponding order of optimization is
\[
    \sup_{\kappa\in\mathcal K_\rho}
    \inf_{(V_H,\Psi_H)\ \text{locally supported}}
    \mathcal E_H(\kappa;V_H),
\]
where the support radius, dimension bound, and local multiplicity are understood to be fixed.

These two quantities differ by more than notation.
In the fixed-visibility model, coefficients that agree on a coefficient-information patch must produce the same anchored local space; this is the common-local-space property in \cref{eq:common-safe-restriction}.
With global coefficient information, a different compactly supported basis may be designed for each coefficient, so the local spaces need not agree across the finite family used in the proof.
The common-local-space argument is then unavailable.

Accordingly, \cref{thm:main} proves an order-one lower bound for fixed-visibility rules, but it neither proves nor disproves the support-only conjecture.
The latter remains open because support locality alone does not restrict how much coefficient information may be encoded in the shape of a local basis function.

\section{Conclusion}

Within the fixed-visibility model, the central scalability question has a negative answer.
No deterministic generalized multiscale finite element space construction can achieve a coefficient-uniform optimal \(\bigO(H)\) error over \(\mathcal K_\rho\) while its support radius, coefficient-information radius, and local multiplicity all remain bounded independently of \(H\).
The failure is stronger than loss of the optimal rate.
The worst-case normalized Galerkin error remains of order one.

The proof identifies why fixed visibility is restrictive.
A finite family of coefficients can coincide on the coefficient-information patches associated with a positive fraction of the mesh while producing more independent local corrector fields than the fixed-dimensional local spaces can approximate.
Strong corrector convergence transfers this local approximation lower bound to an order-one global error for exact solutions with smooth right-hand sides.
Thus a uniformly convergent deterministic construction must allow the support radius, coefficient-information radius, or local multiplicity to grow, or it must obtain coefficient information beyond the prescribed local patches.

The result should not be read as an impossibility theorem for every compactly supported multiscale basis.
Global coefficient information can be encoded in the shape of a locally supported function, and such constructions are not constrained by the common-local-space argument used here.
Whether spaces with fixed support radius and fixed local dimension, designed using global coefficient information, can attain a uniform \(\bigO(H)\) error over the full coefficient class remains unresolved.
The present theorem gives a rigorous error lower bound under fixed visibility and identifies removal of the local coefficient-information restriction as the case not covered by the argument.



\section*{Acknowledgments}
The key construction was developed with assistance from OpenAI Codex.
The author independently verified the proof, prepared the final manuscript, and assumes full responsibility for its mathematical content.
The author thanks Eric Chung, Daniel Peterseim, and Chupeng Ma for helpful discussions on this topic.

\printbibliography

\end{document}